\documentclass[a4paper,11pt]{article}

\usepackage{amsfonts,amssymb,amsmath,amsthm,latexsym,epsf,epsfig,amscd,bbm,stmaryrd}

\usepackage[english]{babel}
\makeatletter \@addtoreset{equation}{section} \makeatother
\makeatletter \@addtoreset{enunciato}{section} \makeatother
\newcounter{enunciato}[section]

\newtheorem{ittheorem}{Theorem}
\newtheorem{itlemma}{Lemma}
\newtheorem{itproposition}{Proposition}
\newtheorem{itdefinition}{Definition}
\newtheorem{itremark}{Remark}
\newtheorem{itclaim}{Claim}
\newtheorem{itfact}{Fact}
\newtheorem{itconjecture}{Conjecture}
\newtheorem{itcorollary}{Corollary}

\newenvironment{theorem}{\addtocounter{enunciato}{1}
\begin{ittheorem}}{\end{ittheorem}}
\newenvironment{lemma}{\addtocounter{enunciato}{1}
\begin{itlemma}}{\end{itlemma}}
\newenvironment{proposition}{\addtocounter{enunciato}{1}
\begin{itproposition}}{\end{itproposition}}
\newenvironment{definition}{\addtocounter{enunciato}{1}
\begin{itdefinition}}{\end{itdefinition}}
\newenvironment{remark}{\addtocounter{enunciato}{1}
\begin{itremark}}{\end{itremark}}

\newenvironment{conjecture}{\addtocounter{enunciato}{1}
\begin{itconjecture}}{\end{itconjecture}}
\newenvironment{corollary}{\addtocounter{enunciato}{1}
\begin{itcorollary}}{\end{itcorollary}}
\newcommand{\be}[1]{\begin{equation}\label{#1}}
\newcommand{\ee}{\end{equation}}
\newcommand{\bl}[1]{\begin{lemma}\label{#1}}
\newcommand{\el}{\end{lemma}}
\newcommand{\br}[1]{\begin{remark}\label{#1}}
\newcommand{\er}{\end{remark}}
\newcommand{\bt}[1]{\begin{theorem}\label{#1}}
\newcommand{\et}{\end{theorem}}
\newcommand{\bd}[1]{\begin{definition}\label{#1}}
\newcommand{\ed}{\end{definition}}
\newcommand{\bp}[1]{\begin{proposition}\label{#1}}
\newcommand{\ep}{\end{proposition}}
\newcommand{\bc}[1]{\begin{corollary}\label{#1}}
\newcommand{\ec}{\end{corollary}}
\newcommand{\bcj}[1]{\begin{conjecture}\label{#1}}
\newcommand{\ecj}{\end{conjecture}}

\def \Z {{\mathbb Z}}
\def \R {{\mathbb R}}

\def \N {{\mathbb N}}

\def \ba {\begin{array}}
\def \ea {\end{array}}

\def\1{\mathbbm{1}}

\begin{document}
\title{Symmetric exclusion as a random environment: 
\\hydrodynamic limits}

\author{\renewcommand{\thefootnote}{\arabic{footnote}}
L.\ Avena \footnotemark[1],
\renewcommand{\thefootnote}{\arabic{footnote}}
T.\ Franco \footnotemark[2]
\\
\renewcommand{\thefootnote}{\arabic{footnote}}
M.\ Jara \footnotemark[3],
\renewcommand{\thefootnote}{\arabic{footnote}}
F.\ V\"ollering \footnotemark[4]}

\footnotetext[1]{Institut f\"ur Mathematik, Universit\"at
Z\"urich, Switzerland. E-mail: luca.avena@math.uzh.ch}

\footnotetext[2]{FCEN, Universidad de Buenos Aires, Argentina and Instituto de Matem\'atica, Universidade Federal da Bahia, Salvador, Brazil. E-mail: tertu@impa.br}

\footnotetext[3]{IMPA, Rio de Janeiro, Brazil. E-mail: mjara@impa.br}
\footnotetext[4]{Mathematisch Instituut, Universiteit Leiden, The Netherlands.
E-mail: fvolleri@math.leidenuniv.nl}

\maketitle

\begin{abstract}
We consider a one-dimensional continuous time random walk with transition rates depending on an 
underlying autonomous simple symmetric exclusion process starting out of equilibrium.
This model represents an example of a random walk in a slowly non-uniform mixing dynamic random environment.
Under a proper space-time rescaling in which the exclusion is speeded up compared to the random walk, 
we prove a hydrodynamic limit theorem for the exclusion as seen by this walk and we derive an ODE describing the macroscopic evolution of the
walk.
The main difficulty is the proof of a replacement lemma for the exclusion as seen from 
the walk without explicit knowledge of its invariant measures.
We further discuss how to obtain similar results for several variants of this model.

\vspace{0.5cm}\noindent

\end{abstract}

\newpage


\section{Introduction}
\subsection{Model and motivation}
\label{s1.1}
Let $\Omega = \{0,1\}^{\mathbb Z}$. Denote by $\eta = \{\eta(z); z \in \mathbb Z\}$
the elements of $\Omega$.
For $\eta \in \Omega$ and $z \in \mathbb Z$, define $\eta^{z,z+1} \in \Omega$ as
\[\eta^{z,z+1}(x) =
\left\{
\begin{array}{cl}
\eta(z+1), & x=z\\
\eta(z), & x=z+1\\
\eta(x), & x \neq z,z+1,
\end{array}
\right.
\]
that is, $\eta^{z,z+1}$ is obtained from $\eta$ by exchanging the occupation variables at $z$ and $z+1$.
Fix $\alpha,\beta \geq 0$ and assume that $\alpha+\beta>0$. Let $\{(\eta_t,x_t);t \geq 0\}$
be the Markov process on the state space $\Omega \times \mathbb Z$ with generator given by
\begin{equation}\begin{aligned}\label{PairGenerator}
L f(\eta,x) &= \sum_{z \in \mathbb Z} \big[f(\eta^{z,z+1},x)-f(\eta,x)\big] 
+\sum_{y\in\{\pm1\}}c_{y}\left(\eta\left(x\right)\right)\big[f(\eta,x+y)-f(\eta,x)\big]\\
&:=
\sum_{z \in \mathbb Z} \big[f(\eta^{z,z+1},x)-f(\eta,x)\big] + \big[\beta + (\alpha-\beta)
\eta(x)\big]\big[f(\eta,x+1)-f(\eta,x)\big]\\
&\quad+ \big[\alpha+(\beta-\alpha)\eta(x)\big] \big[f(\eta,x-1)-f(\eta,x)\big]
,
\end{aligned}\end{equation}
for any local function $f: \Omega \times \mathbb Z \to \mathbb R$.
We interpret the dynamics of the process $\{(\eta_t,x_t);t \geq 0\}$ as follows.
Checking the action of $L$ over functions $f$ which do not depend on $x$, we see that
$\{\eta_t; t \geq 0\}$ has a Markovian evolution, which corresponds to the well-known \emph{simple
symmetric exclusion process} on $\mathbb Z$. Conditioned on a realization of $\{\eta_t; t \geq 0\}$,
the process $\{x_t; t \geq 0\}$ is a random walk that jumps to the left with rate $\beta$ and to
the right with rate $\alpha$ whenever there is a particle at the current position of the random walk (i.e. $\eta(x_t)=1$).
When there is no particle at the current position of the random walk (i.e. $\eta(x_t)=0$), the rates are reversed: it jumps
to the left with rate $\alpha$ and it jumps to the right with rate $\beta$. We say that the simple exclusion
process $\{\eta_t; t \geq 0\}$ is a {\em dynamical random environment} and that $\{x_t; t \geq 0\}$ is a random walk 
in such a dynamical random environment. Note that the fact the random walk has a local drift $\alpha-\beta$ on occupied sites,
and opposite drift $\beta-\alpha$ on vacant sites, creates trapping phenomena typical of random walks in random environments.
This random walk has a tendency to spend a long time around the interface between regions with majority of particles and regions 
with majority of holes. 
As we will show, our proofs and results are still valid for more general choices of the local drifts of the process $\{x_t; t \geq 0\}$, see 
Remark \ref{remark1} and Section \ref{s3.2} below.

Our main results are Theorem \ref{LLN} and Theorem \ref{Hydrolimit} below.
Informally speaking, our results can be formulated as follows. From the hydrodynamic limit theory for the exclusion process it is known that when time is rescaled by $N^2$ and space is rescaled by $1/N$, then the sequence of rescaled exclusion processes converges to the heat equation. We will rescale the random walk driven by the exclusion process differently, namely by $N$ in time and $1/N$ in space (see \eqref{rescalings} for the precise definition). Then we consider the joint sequence of rescaled random walk and rescaled exclusion process and obtain an ODE governing the macroscopic evolution of the random walk limit, and a hydrodynamic limit theorem for the 
exclusion as seen by this walker.

In Section \ref{s3} we will discuss how to derive the same results for several variants of this model.
In order to derive Theorem \ref{Hydrolimit} we need to prove a
so-called \emph{replacement lemma}, Theorem \ref{ReplaceLemma}, in the spirit of \cite{JLS09}. 
Unlike \cite{JLS09}, we do not have explicit knowledge of the invariant
measures for the particle system given by the environment as seen by the walker. 
The latter is a notoriously difficult problem in hydrodynamic limit theory.
However, in our setting, the Bernoulli product measures (which are invariant for the exclusion process) 
turn out to be close to being invariant for the
environment as seen by the walker under the mentioned rescaling.
In fact, we show a crucial estimate in Lemma \ref{DirBound}, which
will allow to control the entropy between the distribution of the evolved environment as seen by the walker  
and the Bernoulli product measures.
This estimate is enough to prove the mentioned \emph{replacement lemma}, Theorem \ref{ReplaceLemma}.
See the paragraph below Theorem \ref{ReplaceLemma} for further explanations.

In recent years, there has been much interest in the study of random walks in random environments. 
See e.g. \cite{S02,Z06} and \cite{AT12,DKL08,RV11} for recent results,
overviews and references in \emph{static} and \emph{dynamic} environments, respectively. 
The aim is to understand the motion of a particle in a material
presenting impurities which is of clear interest for applied purposes.
Despite of the increasing literature on the subject, several basic questions are still open even in one dimension. 
The model we analyse in this paper has been introduced in \cite{AHR10} as a model of a random walk in a Markovian autonomous 
environment with \emph{slow} and \emph{non-uniform} mixing properties due to the fact that space-time correlations in the exclusion
decay slowly and there is not a unique invariant measure. 
Almost no rigorous results are known in the latter setting since the general techniques and results in the literature are 
suitable only for \emph{fast uniform} mixing types of environments.

From a phenomenological point of view, when considering \emph{fast uniform} mixing type of environments 
(e.g. if $\{\eta_t; t \geq 0\}$ evolves according to an independent spin flips dynamics) trapping effects play a minor role.
Indeed the asymptotic behavior of $\{x_t; t \geq 0\}$ in such a setting is qualitatively equivalent to the one of a 
homogeneous random walk, namely, ballisticity in the transient regime, diffusive scaling and exponential decay for large deviations. 
In \cite{AHR10}, the authors considered this model characterized by \eqref{PairGenerator} showing sub-exponential cost for sub-linear 
displacement of $\{x_t; t \geq 0\}$. In other words, despite of the dynamics of the environment, 
this model shows some slow-down phenomena similar to a static random environment (see e.g. \cite{GdH94} and \cite{CGZ00}).
When looking at fluctuations, as suggested by the numerics in \cite{AT12}, non-diffusive behavior is also expected.
Therefore \emph{slowly non-uniform} mixing environments behave very differently from \emph{fast uniform} mixing ones. 

Further rigorous results on this model have been derived in \cite{ASV13} where the authors considered the case where the process 
$\{x_t; t \geq 0\}$ has still two different drifts on occupied and vacant sites of the exclusion but both local drifts in the same 
direction. Under this latter assumption, traps do not play any role and standard diffusive scaling and exponential decay of 
deviations from the typical behavior have been proven by means of a delicate renewal construction.

Unfortunately, the type of rescaling (see \eqref{rescalings}) required in our Theorem \ref{LLN} does not imply the 
law of large numbers for the original random walk considered in \cite{AHR10,AT12}. In fact, to pass from a micro to a macroscopic 
scale, we need that the space-time rescaling of the two processes $\{x_t; t \geq 0\}$ and $\{\eta_t; t \geq 0\}$
are of the same order. In other words, under our space-time rescaling in \eqref{rescalings}, the exclusion jumps much faster than 
the walk preventing ``strong trapping effects''.
Nevertheless, on one hand, our results give a new contribution in the field of random walks in \emph{slowly non-uniform} mixing 
dynamic random environments. On the other hand, they strengthen the connection between random walks in random environments and 
scaling limits of particle systems. In this respect, there is a wide literature on the problem of the tagged particle in 
conservative particle systems, see \cite{LOV00} for a review on this subject. 
The tagged particle can be also interpreted as a random walk in a 
random environment with two main differences with respect to our model: walker and environment are mutually interacting, and the 
invariant measures of the environment as seen by the walker are explicitly known.
Our Theorem \ref{Hydrolimit} is also of interest in itself within the hydrodynamic limit theory because of
this lack of knowledge for the invariant measures of the environment from the point of the walker. 
In the next section we define this process before finally stating the results.

\subsection{The environment as seen by the walker}
\label{1.2}
Consider the process $\{\xi_t; t \geq 0\}$ with values in $\Omega$,
defined by $\xi_t(z)= \eta_t(z+x_t)$. The process $\{\xi_t; t \geq 0\}$ is called the {\em environment as seen by the random walk}.
In other words, $\xi_t=\theta^{x_t}\eta$, where $\theta^z$ denotes the shift operator for $z\in\mathbb Z$.
It turns out that the process $\{\xi_t; t \geq 0\}$ is a Markov process. Its generator is given by
\begin{equation}\begin{aligned}
\mathcal L f(\xi) = \sum_{z \in \mathbb Z} \big[f(\xi^{z,z+1})-f(\xi)\big] + 
\sum_{y\in\{\pm1\}}c_{y}\left(\xi\left(0\right)\right)\big[f(\theta^{y}\xi)-f(\xi)\big]
\end{aligned}\end{equation}
for any local function $f: \Omega \to \mathbb R$, with 
\begin{equation}\label{rates}c_{+1}\left(\xi\left(0\right)\right)=\beta +(\alpha-\beta)\xi(0) \quad\text{and}\quad
c_{-1}\left(\xi\left(0\right)\right)=\alpha+(\beta-\alpha)\xi(0).
\end{equation}
The dynamics of $\{\xi_t; t \geq 0\}$ is the following. 
Particles move according to a simple symmetric exclusion process. Superposed to this dynamics, 
the configuration of particles is shifted to the left or to the right with rates 
corresponding to the jumps of the random walk $\{x_t; t \geq 0\}$.

\subsection{Main results: scaling limits}
\label{s1.3}
In order to state our main results, we first introduce some notation and recall what is known in the literature as the hydrodynamic limit for the simple symmetric 
exclusion process.
Let $u_0: \mathbb R \to [0,1]$ be a piecewise continuous function and let $n \in \mathbb N$ be a scaling 
parameter. We define a probability measure $\mu^n$ in $\Omega$  by
\begin{equation}\label{BernProfile}
\mu^n\big(\eta(z_1)=1,...,\eta(z_\ell)=1\big) = \prod_{i=1}^\ell u_0(z_i/n)
\end{equation}
for any set $\{z_1,...,z_\ell\} \subseteq \mathbb Z$.
We call the sequence of measures $\{\mu^{n}\}_{n\in\mathbb N}$, the \emph{Bernoulli product measures associated to the profile} $u_0$.
Next, let us define the {\em empirical measure} $\{\pi_t^n; t \geq 0\}$ as the measure-valued process given by
\[
\pi_t^n(dx) = \frac{1}{n} \sum_{z \in \mathbb Z} \eta_t^n(z) \delta_{\frac{z}{n}}(dx),
\]
where $\delta_{\frac{z}{n}}(dx)$ denotes the Dirac measure at $\tfrac{z}{n} \in \mathbb R$, and 
$\{\eta_t^n; t \geq 0\}$ denotes the process $\{\eta_{tn^2};t  \geq 0\}$.

The following result is known in the literature as the {\em  hydrodynamic limit} of the simple symmetric exclusion process, see e.g. Theorem 2.1, Chapter 2 in \cite{KL99}.
\begin{proposition}{\bf(Hydrodynamic equation for the simple exclusion)}
\label{p1}
Let $u_0: \mathbb R \to [0,1]$ be a piecewise continuous function and let 
$\{\mu^{n}\}_{n\in\mathbb N}$ be the sequence of Bernoulli product measures associated to the profile $u_0$ as in \eqref{BernProfile}.
Let $\{\eta_t^n; t \geq 0\}$ the process $\{\eta_{tn^2};t  \geq 0\}$ with initial distribution $\mu^n$.
For any $T \geq 0$, the sequence of measure-valued processes $\{\pi_t^n(dx); t \in [0,T]\}_{n \in \mathbb N}$ 
converges to $\{u(t,x)dx; t \in [0,T]\}$ in probability with respect to the $J_1$-Skorohod topology of the 
space of c\`adl\`ag paths $\mathbb D([0,T]; \mathcal M_+(\mathbb R))$, where $\{u(t,x); t \geq 0, x \in \mathbb R\}$ is the 
solution of the Cauchy problem
\begin{equation}\label{HydroSSE}
\left\{
\begin{array}{rll}
\partial_t u(t,x) &= (1/2)\partial^2_{xx} u(t,x), &t \geq 0, x \in \mathbb R\\
u(0,x) &= u_0(x), &x \in \mathbb R.
\end{array}
\right.
\end{equation}
\end{proposition}

Back to our model in \eqref{PairGenerator}, heuristically, it is easy to see that if the average density of particles is different from $\frac{1}{2}$, 
the random walk $\{x_t;t \geq 0\}$ moves ballistically, and therefore the diffusive scaling introduced above
is not the right one for $\{x_t;t \geq 0\}$. A possible way to overcome this fact is to scale the exclusion 
process and the random walk in a different way.
One way to do this is to define the rescaled process $\{(\eta_t^n,x_t^n); t \geq 0\}$ as the Markov process 
generated by the operator
\begin{equation}\begin{aligned}\label{rescalings}
L_n f(\eta,x) &= n^2 \sum_{z \in \mathbb Z} \big[f(\eta^{z,z+1},x)-f(\eta,x)\big]\\ 
	&+ n\sum_{y\in\{\pm1\}}c_{y}\left(\eta\left(x\right)\right)\big[f(\eta,x+y)-f(\eta,x)\big].
\end{aligned}\end{equation}

Under such a rescaling the exclusion jumps faster than the random walk but the motion of the random walk and of the particles 
in the exclusion occur at the same scale allowing for a macroscopic non-trivial description 
of their evolution.
We are finally in shape to state the first result concerning $\{x_t^n;t \geq 0\}$.

\begin{theorem}{\bf (Macroscopic evolution of the random walk)}
\label{LLN} 
Let $u_0: \mathbb R \to [0,1]$ be a piecewise continuous function and let 
$\{\mu^{n}\}_{n\in\mathbb N}$ be the sequence of Bernoulli product measures associated to the profile $u_0$.
Assume that there exists a constant $\rho \in (0,1)$ such that
\begin{equation}
\label{outEqu}
\sup_n \frac{1}{n} \sum_{x \in \mathbb Z} \big|u_0(\tfrac{x}{n})-\rho\big|^2  <+\infty.
\end{equation}
Fix $x_0^n =0$ and $T>0$.
Let $\{\eta_t^n; t \geq 0\}$ be the process $\{\eta_{tn^2};t  \geq 0\}$ with initial distribution $\mu^n$.
Then, for all $t\in[0,T]$:
\[
\lim_{n \to \infty} \frac{x_t^n}{n} = f(t)
\]
in probability, where $\{f(t); t \geq 0\}$ satisfies the ordinary differential equation
\begin{equation}
\label{edo}
\left\{
\begin{array}{rl}
f'(t) &= (\beta-\alpha) \left(1-2u\left(t,f\left(t\right)\right)\right).\\
f(0) & = 0,
\end{array}
\right.
\end{equation}
with $u$ being the solution of \eqref{HydroSSE}.
\end{theorem}

Note that by assuming \eqref{outEqu}, we do not restrict ourselves to the case where the exclusion starts from equilibrium, 
namely, when the initial profile $u_0$ is constantly equal to some fix density $\rho$.
When restricting to this case for the starting density profile, i.e. $u_0\equiv \rho$, Theorem \ref{LLN} reduces to the statement that the random walk observes an averaged homogeneous environment and travels with constant speed given by 
$$f'(t) = (\beta-\alpha) \big(1-2\rho).$$
In other words, when $u_0\equiv \rho$, it is like an homogeneous random walk jumping with probabilities $\beta+\rho(\alpha-\beta)$ and
$\alpha+\rho(\beta-\alpha)$ to the right and to the left, respectively.

The proof of this theorem will be linked to the proof of the following theorem, which represents 
the hydrodynamic limit of the environment as seen by the random walk. 
Let $\{\xi_t^n; t \geq 0\}$ be the rescaled environment as seen by the walk given by $\xi_t^n(x) = \eta_t^n(x+x_t^n)$. 
Its generator ${\mathcal L}_n$ is given by

\begin{equation}\begin{aligned}\label{generatordecomposition}
{\mathcal L}_n f(\xi) &= n^2\sum_{z \in \mathbb Z} \big[f(\xi^{z,z+1})-f(\xi)\big] + 
n\sum_{y\in\{\pm1\}}c_{y}\left(\xi\left(0\right)\right)\big[f(\theta^{y}\xi)-f(\xi)\big]\\
&\quad=:n^2\mathcal L^{\text{ex}}f(\xi) + n \mathcal L^{\text{rw}}f(\xi).
\end{aligned}\end{equation}

Let $\{\hat \pi_t^n(dx); t \geq 0\}$ be the empirical measure associated to the process $\{\xi_t^n; t \geq 0\}$:
\begin{equation}
\label{EmpMeas}
\hat \pi_t^n(dx) = \frac{1}{n} \sum_{z \in \mathbb Z} \xi_t^n(z) \delta_{\frac{z}{n}}(dx).
\end{equation}
We have the following:
\begin{theorem}{\bf(Hydrodynamic equation for the environement as seen by the walker)}
\label{Hydrolimit} Fix $T \geq 0$, under the assumptions of Theorem \ref{LLN}, the sequence of measure-valued 
processes $\{\hat \pi_t^n(dx);t \in [0,T]\}_{n\in \mathbb N}$ converges to $\{\hat u(t,x) dx; t \in [0,T]\}$ 
in probability with respect to the $J_1$-Skorohod topology of $\mathbb D([0,T]; \mathcal M_+(\mathbb R))$, 
where $\{\hat u(t,x); t \geq 0, x \in \mathbb R\}$ is the solution of the Cauchy problem
\begin{equation}\label{PDE}
\left\{
\begin{array}{rll}
\partial_t \hat u(t,x) &= (1/2)\partial^2_{xx} \hat u(t,x) + (\beta-\alpha)\big(1-2\hat u(t,0)\big) 
\partial_x \hat u(t,x), &t \geq 0, x \in \mathbb R\\
\hat u(0,x) &= u_0(x), &x \in \mathbb R.
\end{array}
\right.
\end{equation}
\end{theorem}

Theorem \ref{Hydrolimit} says that the effect of the random walk on the macroscopic evolution of the process 
$\{\xi_t^n; t \geq 0\}$ results into a transport term.
This term reflects the fact that microscopically every time the random walk jumps to the right, the exclusion gets shifted to the left, and vice versa.
The speed of this transport mechanism is the speed of the walker that is dependent on $\hat u(t,0)=u(t,f(t))$, which represents
the density observed by $\{x_t^n; t \geq 0\}$.

Let us further stress that Theorem \ref{Hydrolimit} has its own interest within hydrodynamic theory of particle systems.
In fact, it provides the hydrodynamic equation for the process in \eqref{generatordecomposition} 
for which the invariant measures are not explicitly known and which can be seen as a non-trivial perturbation of the exclusion 
process.

\begin{remark}\label{remark1}
{{\bf (More general jump rates)}}
The results of this article remain true for rather arbitrary jump
rates. The particular choice of $c_{\pm1}$ made in \eqref{rates} allows us to
keep the presentation and the proofs at a reasonable level of
technicality.
In section \ref{s3.2} we state our theorems in more generality and we show where and how the proof presented in Section \ref{s2} 
has to be adapted.
\end{remark}

\begin{remark}\label{remark3}{{\bf (Beyond the macroscopic speed)}}
In this paper we focus on the macroscopic speed of $\{x_t; t \geq 0\}$. 
It is natural to ask about fluctuations and large deviations for 
our model. It turns out that these two further questions are very delicate due to the lack of knowledge of the 
equilibrium measures of the process $\xi_t$.
In other words, we need a finer control on the properties 
of the additive functional of the process in \eqref{addfun}.
We plan to address these questions in future works.
Another direction which we are currently working on are conservative systems of ballistic particles.
In such a setting the rescaling of the environment naturally matches the rescaling for the random walk.    
\end{remark}

The rest of the paper is organized as follows.
Section \ref{s2} is devoted to the proof of Theorems \ref{LLN} and \ref{Hydrolimit}.
The proofs are based on the so-called \emph{entropy method} and it will be divided in 
the following three main steps. \emph{Step 1):}
In Section \ref{s2.1} we show that the analysis of the random walk can be reduced to the study of the additive functional $\int_0^t \big(1-2\xi_s^n(0)\big)ds$, and we prove Theorem \ref{ReplaceLemma}, a so-called \emph{local replacement lemma}.
This theorem allows to approximate the additive functional above in terms of the empirical measure in \eqref{EmpMeas}.
{\em Steps 2) and 3):} Sections \ref{s2.2} and \ref{s2.3} deal with two necessary topological issues to show a weak convergence result 
of the type of Theorem \ref{Hydrolimit}, namely, tightness
of the sequence of considered processes, and the characterizations of the corresponding limiting points, respectively.
In Section \ref{s3}, we discuss possible variants or generalizations of the model presented 
in Section \ref{s1.1}, for which the same techniques can be adapted. In particular in Section \ref{s3.1}, we show how to derive the same type of results when considering the so-called speed-change exclusion as a dynamic environment. While Sections \ref{s3.2} and \ref{s3.3} focus on general jump rates for the random walk.

\section{The entropy method}
\label{s2}
For $\rho \in [0,1]$, let us denote by $\nu_\rho$ the product Bernoulli measure of density $\rho$, 
that is, $\nu_\rho$ is a probability measure in $\Omega$ such that
\[
\nu_\rho\big(\eta(z_1)=1,...,\eta(z_\ell)=1\big) = \rho^\ell
\]
for any set $\{z_1,...,z_\ell\} \subseteq \mathbb Z$. The measures $\nu_\rho$ are invariant with respect to 
the evolution of the process $\{\eta_t; t\geq 0\}$. The same is
no longer true for the environment as seen by the walk, that is, the measures $\nu_\rho$ are {\em not left invariant} 
by the evolution of $\{\xi_t; t \geq 0\}$. In fact, if $\alpha > \beta$,
the process $x_t$ likes to be around the interface between regions with majority of particles and regions 
with majority of holes. This fact is the main difficulty in order to prove
Theorems \ref{LLN} and \ref{Hydrolimit}. As we will see, it turns out that the measure $\nu_\rho$ is close 
to be invariant in some sense.

\subsection{A local replacement lemma}
\label{s2.1}
Let $\{\xi_t^n; t \geq 0\}$ be the rescaled environment as seen by the walk. 
For simplicity, we fix $T >0$ and we consider the evolution of the process $\xi_t^n$ up to time $T$. 

Let us fix some notation.
Let $\mathbb D([0,T]; \Omega)$ be the space of c\`adl\`ag trajectories from $[0,T]$ to $\Omega$. 
Denote by $\mathbb P^n$ the distribution in $\mathbb D([0,T]; \Omega)$ of the process $\{\xi_t^n; t \in [0,T]\}$ 
(with initial distribution $\mu^n$) and by $\mathbb E^n$ the expectation with respect to $\mathbb P^n$. 
We denote by $\mathbb P_\rho^n$ the distribution of the process $\{\xi_t^n; t \in [0,T]\}$ starting from the distribution
$\nu_\rho$ and by $\mathbb E_\rho^n$ the expectation with respect to $\mathbb P_\rho^n$. 
We will use the same notations for the process $\{(\eta_t^n,x_t^n); t \in [0,T]\}$ with $x_0^n=0$.
Positive unspecified constants will be denoted by $c$.

We can recover the position $x_t^n$ of the random walk in dynamical random environment looking 
at the signed number of shifts of the process $\{\xi_t^n;t \geq 0\}$ up to time $t$. 
Since the number of shifts is a compound Poisson process, we can write
\begin{equation}\label{Martingale}
\frac{x_t^n}{n} = \widetilde{ { M}_t^n} + (\beta-\alpha) \int_0^t \big(1-2\xi_s^n(0)\big) ds,
\end{equation}
where $\widetilde{ { M}_t^n}$ is a martingale of quadratic variation $\langle\widetilde{ { M}_t^n}\rangle = \frac{(\alpha+\beta)t}{n}$.
In particular, $\sup_{t \leq T} |\widetilde{ { M}_t^n}|$ converges to 0 in probability as $n \to \infty$ for any $T \geq 0$. 
Therefore, in order to obtain a law of large numbers for the process $\{x_t^n; t \geq 0\}$ it is enough to obtain a law of 
large numbers for the integral
\begin{equation}\label{addfun}
\int_0^t \big(1-2\xi_s^n(0)\big) ds.
\end{equation}
What we will first prove is the following theorem.
\begin{theorem}{\bf(Local replacement lemma)}
\label{ReplaceLemma}
Under the assumptions of Theorem \ref{Hydrolimit},
\footnote{
Here and below, for $\epsilon >0$ and $n \in \N$, we use $\epsilon n$ indistinctly for the real number $\epsilon n$ and 
for its integer part $\lfloor \epsilon n\rfloor$.
}
\begin{equation}
\label{localRL}
\lim_{\epsilon \to 0} \limsup_{n \to \infty}
\mathbb E^n\left[\left|\int_0^t \left(\xi_s^n(0) - \frac{1}{\epsilon n} \sum_{z=1}^{\epsilon n} \xi_s^n(z)\right)ds
\right|\right]=0.
\end{equation}
\end{theorem}

This theorem is called \emph{local replacement lemma} since allows to ``replace'' the local function $\xi_s^n(0)$ by the spatial average in \eqref{localRL} which can be rewritten in terms of the empirical measure $\hat \pi_s^n \in \mathcal M_+(\mathbb R)$ in \eqref{EmpMeas} as 
$$\frac{1}{\epsilon n} \sum_{z=1}^{\epsilon n} \xi_s^n(z)=\int \epsilon^{-1}\mathbbm{1}_{(0,\epsilon]}(x) \hat \pi_s^n(dx).$$
Such a replacement is a crucial step into the proof of the hydrodynamic limit of general diffusive particle systems and in fact, 
we will use it in the final argument of Section \ref{s2.3} to get our main theorems.
Let us remark that a similar statement is proved in \cite{JLS09} in the context of a zero-range process as seen by a tagged particle. 
It turns out that the process considered in \cite{JLS09} has an invariant measure of product form, 
and by this reason the proof does not applies straightforwardly to our model.
We will see below that an estimate on the Dirichlet form (see Lemma \ref{DirBound} below) is all that we need in order to adapt the proof of \cite{JLS09} 
to our setting and prove Theorem \ref{ReplaceLemma}.
The motivation for estimating the Dirichlet form is the following.
When a Markov chain has a unique invariant measure, 
the (relative) entropy of the distribution of the chain with respect to the invariant measure is 
decreasing and vanishing in time. 
The speed at which the entropy decreases can be controlled by the Dirichlet form associated to the invariant measure. 
The same is true for general Markov processes, although the 
entropy with respect to some invariant measure may not go to zero, it is still decreasing, 
and the speed is still controlled by the associated Dirichlet form. 
For the process $\{\xi_t; t \geq 0\}$, the measure $\nu_\rho$ is not invariant, 
and therefore the entropy with respect to $\nu_\rho$ is not necessarily decreasing. 
This reflects on the fact that $\langle\sqrt f, \mathcal L_n \sqrt f\rangle$ may not be negative.
However, in our model the entropy does not grow too much. 
In fact, under the assumptions of Theorem \ref{LLN}, as we show at the beginning of the proof of Theorem \ref{ReplaceLemma} below,
the entropy between $\mu^n$ and $\nu_\rho$ is of order $n$.
According to the entropy method introduced in \cite{GPV88}, this entropy bound should be enough in order to prove the {\em replacement lemma}, and this is what we do next.
We first introduce the Dirichlet form and prove the mentioned estimate in Lemma \ref{DirBound}, then by using this estimate we prove Theorem \ref{ReplaceLemma}.

Fix $\rho \in (0,1)$, which we assume to be the one in the assumptions 
of Theorem \ref{LLN}. Let $f: \Omega \to \mathbb R$ be a density with respect to $\nu_\rho$,
that is, $f(\eta) \geq 0$ for every $\eta \in \Omega$ and $\int f d\nu_\rho =1$. 
We define the {\em Dirichlet form} of $f$ associated to  $\mathcal L^{\text{ex}}$ 
(see \eqref{generatordecomposition}) with respect to $\nu_\rho$ as
\begin{equation}\label{Dirichlet}
\mathcal D(f) = \frac{1}{2} \sum_{z \in \mathbb Z} \int \big(\sqrt{f(\eta^{z,z+1})}-\sqrt{f(\eta)}\big)^2 \nu_\rho(d\eta).
\end{equation}
Let us denote by $\langle\cdot,\cdot\rangle$ the inner product in $L^2(\nu_\rho)$.
It is standard to check that
\begin{equation}\label{dirichlet}
\mathcal D(f) = \langle\sqrt f,-\mathcal L^{\text{ex}} \sqrt f\rangle. 
\end{equation}

We now show a crucial bound on $\langle\sqrt f, \mathcal L_n \sqrt f\rangle$.

\begin{lemma}{\bf(Dirichlet form estimate)}
\label{DirBound}
For any density $f:\Omega \to \mathbb R$ with respect to $\nu_\rho$,
\[
\langle\sqrt f, \mathcal L_n \sqrt f\rangle \leq -n^2 \mathcal D(f) + 2n|\alpha-\beta|.
\]
\end{lemma}
\begin{proof}
As in equation \eqref{generatordecomposition}, 
we can split $\mathcal L_n$ into two parts: $n^2 \mathcal L^{\text{ex}}$, which is the resclaed part of the generator 
corresponding to jumps of the particles, and $n \mathcal L^{\text{rw}}$, which corresponds to rescaled jumps of the 
random walk.
By \eqref{dirichlet}, the part of the generator corresponding to jumps of the particles satisfies $\langle\sqrt f, n^2 \mathcal L^{\text{ex}}\sqrt f\rangle = -n^2 \mathcal D(f)$. 
Therefore, we just need to prove that
\[
\langle\sqrt f, \mathcal L^{\text{rw}} \sqrt f\rangle \leq 2|\alpha -\beta|
\]
for any density $f$. Notice that $\langle\sqrt f, \sqrt f\rangle=1$ and notice that also $\langle\sqrt{\theta^1f}, \sqrt{\theta^1f}\rangle=1$. 
This last identity follows from the invariance of the measure $\nu_\rho$ under spatial shifts. 
By the same reason,  $\langle\sqrt{\theta^{-1} f}, \sqrt{\theta^{-1} f}\rangle=1$. By the Cauchy-Schwarz inequality,
\[
\langle\sqrt f, \sqrt{\theta^1f}\rangle \leq \frac{1}{2}\big(\langle\sqrt f,\sqrt f\rangle + \langle\sqrt{\theta^1f}, \sqrt{\theta^1f}\rangle\big) =1.
\]
Therefore,
\begin{equation}\begin{aligned}\notag
\langle\sqrt f, \big(\beta +(\alpha-\beta)\xi(0)\big)
	&\big(\sqrt{\theta^1f}-\sqrt f\big)\rangle \\
	&\leq \max\{\beta+(\alpha-\beta)\xi(0)\} \langle\sqrt f , \sqrt{\theta^1f}\rangle \\
	& \quad-\min\{\beta+(\alpha-\beta)\xi(0)\}\langle\sqrt f, \sqrt f\rangle\\
	&\leq  \max\{\beta+(\alpha-\beta)\xi(0)\}-\min\{\beta+(\alpha-\beta)\xi(0)\}\\
	&\leq |\alpha -\beta|.
\end{aligned}\end{equation}
The same reasoning allows to bound the other term in $\langle\sqrt f, \mathcal L^{\text{rw}} \sqrt f\rangle$, see \eqref{generatordecomposition}, 
which ends the proof of the lemma.
\end{proof}

We are now in shape to prove Theorem \ref{ReplaceLemma}, using Lemma \ref{DirBound}.

\begin{proof}[Proof of Theorem \ref{ReplaceLemma}]
Let us denote by $H(\mu|\nu)$ the entropy between two given measures.
First we observe that $H(\mathbb P^n|\mathbb P_\rho^n)= H(\mu^n|\nu_\rho) \leq c n$. 
The first equality is an easy general fact about Markov chains.
The second inequality follows from the structure of the measures $\mu^n$ and $\nu_\rho$ and the assumption in \eqref{outEqu}.
By the entropy inequality, see Proposition 8.1 in \cite{KL99}, for any function $V: \Omega \to \mathbb R$ and any $\gamma >0$,
\begin{equation}
\label{ec1}
\mathbb E^n\Big[\Big|\int_0^t V(\xi_s^n)ds\Big|\Big]
	\leq \frac{H(\mathbb P^n|\mathbb P_\rho^n)}{\gamma n} + \frac{1}{\gamma n} \log \mathbb E^n_\rho\big[e^{\gamma n\big|\int_0^t V(\xi_s^n)ds\big|}\big] .
\end{equation}
The first term on the right-hand side of \eqref{ec1} is then bounded by $\frac{c}{\gamma}$. 
Using the estimate $e^{|a|} \leq e^a+ e^{-a}$, we can get rid of the absolute value on the second term of the 
right-hand side of \eqref{ec1}, at the cost of estimating two expectations, one involving $V$ and another involving 
$-V$. Therefore, it suffices to show that 
\begin{equation}
\label{ec2}
\lim_{\gamma \to \infty} \limsup_{\epsilon \to 0} \limsup_{n \to \infty}
\frac{1}{\gamma n} \log \mathbb E^n_\rho\big[e^{\gamma n\int_0^t V(\xi_s^n)ds}\big]=0,
\end{equation}
where
\[
V(\xi) = \pm\Big(\xi(0)- \frac{1}{\epsilon n} \sum_{z = 1}^{\epsilon n} \xi(z)\Big).
\]
By Feynman-Kac formula, we can express the expectation in \eqref{ec2} in terms of the semigroup associated to the operator $\mathcal L_n+ \gamma n V$. 
After some computations based on spectral theory (see Lemma A.1.7.2 in \cite{KL99}), we can obtain the bound
\[
\frac{1}{\gamma n} \log \mathbb E^n_\rho\big[e^{\gamma n\int_0^t V(\xi_s^n)ds}\big]
	\leq t \sup_{f} \Big\{ \langle V,f\rangle + \frac{1}{\gamma n} \langle\sqrt f, \mathcal L_n \sqrt f\rangle\Big\},
\]
where the supremum is taken over all the densities $f$ with respect to $\nu_\rho$. By Lemma \ref{DirBound}, the supremum above is bounded by
\[
\frac{2t|\alpha-\beta|}{\gamma} + t \sup_{f} \Big\{ \langle V,f\rangle - \frac{n}{\gamma}\mathcal D(f)\Big\}.
\]
Therefore, in order to prove the theorem, we only need to prove that
\begin{equation}\begin{aligned}\label{eq:replacement}
\lim_{\gamma \to \infty} \limsup_{\epsilon \to 0} \limsup_{n \to \infty}  \sup_{f} \Big\{ \langle V,f\rangle - \frac{n}{\gamma}\mathcal D(f)\Big\} =0.
\end{aligned}\end{equation}
Notice that
\[
\xi(0)- \frac{1}{\epsilon n} \sum_{z = 1}^{\epsilon n} \xi(z) = \sum_{z=1}^{\epsilon n}  \frac{\epsilon n-z+1}{\epsilon n}\big(\xi(z-1)-\xi(z)\big).
\]
We first estimate $\langle\xi(z-1)-\xi(z),f\rangle$. Performing the change of variables $\xi \to \xi^{z,z+1}$, we see that
\[
\langle\xi(z-1)-\xi(z),f\rangle = \int \xi(z)\big(f(\xi^{z,z+1})-f(\xi)\big) \nu_\rho(d \xi).
\]
Let us define
\[
\mathcal D_z(f) = \int \big(\sqrt{f(\xi^{z,z+1})}-\sqrt{f(\xi)}\big)^2 \nu_\rho(d\xi),
\]
in such a way that $\mathcal D(f) = \frac{1}{2}\sum_z \mathcal D_z(f)$.
Writing
\[
f(\xi^{z,z+1})-f(\xi) = \big\{\sqrt{f(\xi^{z,z+1})}-\sqrt{f(\xi)}\big\}\big\{\sqrt{f(\xi^{z,z+1})}+\sqrt{f(\xi)}\big\},
\]
using the Young inequality and the fact that $\int f d\nu_\rho =1$ and $0\leq \xi(z) \leq 1$, we see that
\[
\langle\xi(z-1)-\xi(z),f\rangle \leq \frac{1}{2\lambda_z} \mathcal D_z(f) + 2\lambda_z
\]
for any $\lambda_z>0$. Choosing
\[
\lambda_z = \frac{\gamma}{n}\frac{\epsilon n -z +1}{\epsilon n},
\]
we obtain 
\[
\langle V,f\rangle - \frac{n}{\gamma} \mathcal D(f) \leq 
\frac{2\gamma}{\epsilon^2 n^3}\sum_{z=1}^{\epsilon n} (\epsilon n-z+1)^2 \leq c \gamma \epsilon
\]
for any density $f$, which gives \eqref{eq:replacement} and concludes the theorem.
\end{proof}

\subsection{Tightness}
\label{s2.2}
The proof of a theorem like Theorem \ref{Hydrolimit} is usually performed following the classical three-steps 
procedure to prove weak convergence of probability measures in Polish spaces. 
The first step is to prove tightness of the corresponding sequence of distributions with respect to some 
properly chosen topology. 
The second step is to prove that any limit point (they exists due to tightness 
and Prohorov's Theorem) satisfies a convenient set of properties. 
The third step is to prove that there exists at most one probability measure on the corresponding Polish space
satisfying those properties. 
Then, an abstract topology result, namely that any relatively compact set with a unique limit point is a 
converging sequence shows that the sequence of distributions converges in distribution with respect to the 
already chosen topology. In the case of the hydrodynamic limit stated in Theorem \ref{Hydrolimit}, convergence 
in probability is readily obtained, since the limit is non-random.
In this section we show the first of these steps: the tightness. 
Before entering into the details, we do a small detour to explain the choice of topology and some facts 
associated to this choice.

Let $\mathcal B(\mathbb R)$ denote the family of Borel sets of $\mathbb R$, that is, the smallest $\sigma$-algebra formed 
by subsets of $\mathbb R$ containing all the open sets of $\mathbb R$.
We say that a measure $\pi(dx)$ defined in $\mathbb R$ is {\em non-negative} if $\pi(A) \geq 0$ for any 
$A \in \mathcal B(\mathbb R)$. 
We say that the measure $\pi(dx)$ is a {\em Radon} measure if $-\infty<\pi(K)<+\infty$ for any compact set 
$K \subseteq \mathbb R$.
The empirical measure $\{\hat \pi_t^n(dx); t  \in [0,T]\}$, is a random, non-negative, Radon measure. 
Let us denote by  $\mathcal M_+(\mathbb R)$ the set of non-negative, Radon measures in $\mathbb R$. 
The weak topology in $\mathcal M_+(\mathbb R)$ is defined in the following way. 
Let $\mathcal C_c(\mathbb R)$ denote the set of functions $f:\mathbb R \to \mathbb R$ which are continuous and of bounded support. 
We say that a sequence  $\{\pi^n;n \in \mathbb N\}$ in $\mathcal M_+(\mathbb R)$ converges to a measure $\pi \in \mathcal M_+(\mathbb R)$
if for any  $f\in \mathcal C_c(\mathbb R)$,
\[
\lim_{n \to \infty} \int f d\pi^n = \int f d\pi.
\]
The space $\mathcal M_+(\mathbb R)$ turns out to be a Polish space with respect to the weak topology. 
In fact, there exists a sequence $\{f_\ell;\ell \in  \mathbb N\}$ of functions in $\mathcal C_c(\mathbb R)$ such that the 
distance $d: \mathcal M_+(\mathbb R) \times \mathcal M_+(\mathbb R)$ defined by
\[
d(\pi,\pi') = \sum_{\ell \in \mathbb N} \frac{1}{2^\ell} \min\Big\{\Big|\int f_\ell d(\pi-\pi')\Big|,1\Big\}
\]
is a metric for the weak topology in $\mathcal M_+(\mathbb R)$. We can assume that the functions $\{f_\ell; \ell \in \mathbb N\}$ 
are infinitely differentiable. We refer the reader to Section 4 of \cite{KL99} for the latter statement as well as for the following topological considerations. In order to simplify the notation, we will write $\pi(f) = \int f d\pi$ for 
$\pi \in \mathcal M_+(\mathbb R)$ and $f \in \mathcal C_c(\mathbb R)$. Let us consider the following topology in $\mathcal C_c(\mathbb R)$.
 We say that the sequence $\{f_n; n \in \mathbb N\}$ of functions in $\mathcal C_c(\mathbb R)$ converges to $f \in \mathcal 
C_c(\mathbb R)$ if two things happen:
\begin{itemize}
\item[i)] there exists a compact $K$ such that the support of $f_n$ is contained in $K$ for any $n \in \mathbb N$,
\item[ii)]
\[
\lim_{n \to \infty} \sup_{x \in \mathbb R} |f_n(x)-f(x)| =0.
\]
\end{itemize}

It turns out that $\mathcal M_+(\mathbb R)$ is the dual of $\mathcal C_c(\mathbb R)$ with respect to this topology, and moreover, 
the weak topology in $\mathcal M_+(\mathbb R)$ is the weak-$*$ topology of $\mathcal M_+(\mathbb R)$ associated to this duality. 
This fact provides us with a very simple tightness criterion for measure-valued processes:

\begin{proposition}
\label{p2}
Let $\{\pi_t^n(dx); t \in [0,T]\}_{n \in \mathbb N}$ be a sequence of measure-valued processes with trajectories 
in the space $\mathbb D([0,T]; \mathcal M_+(\mathbb R))$. The sequence of processes $\{\pi_t^n(dx); t \in [0,T]\}_{n \in \mathbb N}$
is tight with respect to the $J_1$-Skorohod topology of $\mathbb D([0,T];\mathcal M_+(\mathbb R))$ if and only if the sequence of real-valued processes $\{\pi_t^n(f_\ell); t \in [0,T]\}_{n \in \mathbb N}$ is tight with respect to the $J_1$-Skorohod topology of $\mathbb D([0,T]; \mathbb R)$ for any $\ell \in \mathbb N$. If, for any $\ell \in \mathbb N$, any limit point of $\{\pi_t^n(f_\ell); t \in [0,T]\}_{n \in \mathbb N}$ is supported on the space $\mathcal C([0,T];\mathbb R)$ of continuous functions, then any limit point of $\{\pi_t^n(dx); t \in [0,T]\}_{n \in \mathbb N}$ is supported on the space $\mathcal C([0,T]; \mathcal M_+(\mathbb R))$.
\end{proposition}
Roughly speaking, this proposition is saying that the verification of tightness for measure-valued processes 
can be reduced to the verification of tightness for real-valued processes. This property holds true for any 
dual of a Polish space equipped with the weak-$*$ topology.
For a proof of this Proposition in the context of measure-valued processes, see Section 4 of \cite{KL99}.
The proof can be easily adapted to any dual of a Polish space.

Now we are in position to state the tightness results for the random walk and the environment as viewed by the walk:

\begin{theorem}
\label{t4}
The sequence of processes $\{\hat \pi_t^n; t \in [0,T]\}_{n \in \mathbb N}$ is tight with respect to the 
$J_1$-Skorohod topology of $\mathbb D([0,T]; \mathcal M_+(\mathbb R))$. Moreover, any limit point of 
$\{\hat \pi_t^n; t \in [0,T]\}_{n \in \mathbb N}$ is supported on the set $\mathcal C([0,T]; \mathcal M_+(\mathbb R))$ of 
continuous trajectories. The sequence of real-valued processes $\{n^{-1} x_t^n; t \in [0,T]\}_{n \in \mathbb N}$ 
is tight with respect to the $J_1$-Skorohod topology of $\mathbb D([0,T]; \mathbb R)$. 
Moreover, any limit point of $\{n^{-1} x_t^n; t \in [0,T]\}_{n \in \mathbb N}$ is supported on the set 
$\mathcal C([0,T]; \mathbb R)$.
\end{theorem}

\begin{proof}
Let us recall Dynkin's formula for a function of a Markov process: for any local function $F:\Omega \to \mathbb R$, 
the process
\[
F(\xi_t^n) - F(\xi_0^n) -\int_0^t \mathcal L_n F(\xi_s^n)ds
\]
is a mean-zero martingale. The quadratic variation of this martingale can also be computed in terms of the generator $\mathcal L_n$ and it is given by
\[
\int_0^t \mathcal L_n F(\xi_s^n)^2-2F(\xi_s^n)\mathcal L_n F(\xi_s^n) ds = \int_0^t \left[\mathcal L_n\left(F-F(\xi^n_s)\right)^2\right](\xi^n_s)\,ds .
\]
Let $f \in \mathcal C_c(\mathbb R)$ be a smooth function. Let us define the discrete gradient(s) and discrete Laplacian of $f$ by
\begin{equation}\begin{aligned}\notag
\Delta_n f(\tfrac{x}{n}) &= n^2\big[f(\tfrac{x+1}{n})+f({\tfrac{x-1}{n}})-2f(\tfrac{x}{n})\big],\\
\nabla^n_-\! f(\tfrac{x}{n}) &= n\big[f(\tfrac{x}{n})-f(\tfrac{x-1}{n})\big],\\
\nabla^n_+\! f(\tfrac{x}{n}) &= n\big[f(\tfrac{x+1}{n})-f(\tfrac{x}{n})\big].
\end{aligned}\end{equation}
Taking $F(\xi_t^n) = \hat \pi_t^n(f)$, we see that
\begin{equation}\begin{aligned}\notag
{M}_t^n(f) &= \hat \pi_t^n(f) -\hat \pi_0^n(f) - 
\int_0^t \big\{\hat \pi_s^n(\Delta_n f)- (\beta-\alpha)(1-2\xi_s^n(0))\hat \pi_s^n(\nabla_-^n\! f) \\
&\qquad\qquad -\frac1n [\alpha+(\beta-\alpha)\xi^n_s(0)]\hat\pi^n_s(\Delta_n f) \big\}ds
\end{aligned}\end{equation}
is a martingale. 
Note that the last term comes from the difference between $\nabla^n_+f-\nabla^n_-f$. 
Its quadratic variation is

\begin{equation}\begin{aligned}\notag
\langle {M}_t^n(f)\rangle &= \frac{1}{n}\int_0^t \frac{1}{n} \sum_{x \in \mathbb Z} \left(\xi_s^n(x)-\xi_s^n(x\!-\!\!1)\right)^2\left
(\nabla_-^n\! f(\tfrac{x}{n})\right)^2 \\
&\quad +\frac{\alpha+\beta}{n} \int_0^t \left(\hat \pi_s^n(\nabla_-^n\! f)\right)^2ds + E_n,
\end{aligned}\end{equation}
where $E_n$ is a lower order term which captures the influence of $\nabla^n_+f-\nabla^n_-f$ and is given by
\[ \frac1n \int_0^t (\beta\xi_s^n(0)+\alpha(1-\xi_s^n(0))[
(\hat\pi_s^n(\nabla^n_+f))^2-(\hat\pi_s^n(\nabla^n_-f))^2 ] ds.\]
Notice that the occupation variables can only assume the values $0,1$. Since $f$ is infinitely differentiable 
and of compact support, there exists a constant $C(f)>0$ such that
\[
\big|\langle {M}_t^n(f)\rangle\big| \leq \frac{C(f)t}{n}
\]
for any $t \in [0,T]$ and any $n \in \mathbb N$. Note that $\mathbb E^n[{M}_t^n(f)^2] = \mathbb E^n[\langle {M}_t^n(f)\rangle]$. 
Therefore, by Doob's inequality we have
\[
\mathbb P^n\Big[\!\sup_{0 \leq t \leq T} \!\big|{M}_t^n(f)\big| \geq \epsilon\Big] \leq \frac{4C(f)T}{\epsilon^2 n}
\]
for any $n \in \mathbb N$. We conclude that the process $\{{M}_t^n(f); t \geq 0\}$ converges to 0 in probability 
(and therefore in distribution) with respect to the uniform topology in $\mathbb D([0,T];\mathbb R)$. 
Notice that $\hat \pi_0^n(f) = \pi_0^n(f)$. Therefore, by hypothesis $\hat \pi_0^n(f)$ converges in probability 
(and in distribution) to $\int u_0(x) f(x) dx$. 
These two convergences reduce the proof of tightness of $\{\hat \pi_t^n(f); t\in [0,T]\}$ to the proof of 
tightness of the integral term
\[
\mathcal I_t^n(f)=\int_0^t \big\{\hat\pi^n_s(\Delta_n f)- (\beta-\alpha)(1-2\xi_s^n(0)) \hat \pi_s^n(\nabla_-^n\! f)\big\}ds.
\]
But, using again the boundedness of $\xi_s^n(0)$,
\[
\big|\mathcal I_t^n(f)-\mathcal I_s^n(f)\big| \leq \frac{|t-s|}{n} \sum_{x\in \mathbb Z} \left(
\big|\Delta_n f(\tfrac{x}{n})\big| + \big|\nabla_-^n\!f(\tfrac{x}{n})\big|\right),
\]
and by the smoothness of $f$ we conclude that the sequence of processes 
$\{\mathcal I_t^n; t \in [0,T]\}_{n \in \mathbb N}$ is uniformly Lipschitz (uniformly in $n$ and $t$), 
and in particular it is tight with respect to the uniform topology of $\mathcal C([0,T];\mathbb R)$. 
This proves two things: first, the sequence $\{\hat \pi_t^n(f); t \in [0,T]\}_{n \in \mathbb N}$ is tight; 
and second, every limit point of  $\{\hat \pi_t^n(f); t \in [0,T]\}_{n \in \mathbb N}$ is supported on continuous 
trajectories. By Proposition \ref{p2}, the sequence of measure-valued processes $\{\hat \pi_t^n; t \in [0,T]\}_{n \in \mathbb N}$ is tight with respect to the $J_1$-Skorohod topology on $\mathbb D([0,T]; \mathcal M_+(\mathbb R))$ and any limit point is supported on continuous trajectories.

Now we turn into the proof of tightness of the sequence $\{\frac{x_t^n}{n}; t \in [0,T]\}$. 
This is actually simpler. In fact, the process
\[
\widetilde{ { M}_t^n} :=\frac{x_t^n}{n} - (\beta-\alpha) \int_0^t (1-2\xi_s^n(0))ds
\]
is a martingale of quadratic variation $\frac{(\alpha+\beta)t}{n}$. As above, 
by Doob's inequality the martingale converges to 0 in probability with respect to the uniform topology on 
$\mathbb D([0,T]; \mathbb R)$. Again as above, the integral term is uniformly Lipschitz, both in $t$ and $n$. 
In fact, the Lipschitz constant is bounded above by $|\beta-\alpha|$. 
Therefore, $\{\frac{x_t^n}{n}; t \in [0,T]\}$ is tight with respect to the $J_1$-Skorohod topology on 
$\mathbb D([0,T];\mathbb R)$ and any limit point is supported on continuous trajectories. 
In fact, we can say a little bit more about the limit points: they are supported on {\em Lipschitz functions} 
of Lipschitz constant bounded by $|\beta - \alpha|$.
\end{proof}

\subsection{Characterization of limit points: proofs of Theorems \ref{LLN} and \ref{Hydrolimit}}
\label{s2.3} We are now ready to finish our proofs.
We first observe that a vector of random processes is tight if and only if each coordinate is tight. 
Therefore, by Theorem \ref{t4}, the triple $\{(\pi_t^n,\hat \pi_t^n,\frac{x_t^n}{n}); t \in [0,T]\}_{n \in \mathbb N}$ is tight. 
Let $\{(u(t,x)dx, {\hat u}_t(dx), f(t)); t \in [0,T]\}$ be a limit point of the triple and let $n'$ be the 
subsequence of $n$ for which the triple converges to that limit point. 
Remind that we already know that $u(t,x)dx$ is the solution of the heat equation, and in particular it is 
deterministic, but in principle $f(t)$ and ${\hat u}_t(dx)$ may be random. 
A first observation is that ${\hat u}_t(dx)$ has a density with respect to Lebesgue measure.
This is an easy consequence of the boundedness of $\xi_s^n$. 
In fact, for any closed, non-empty interval 
$A \subseteq \mathbb R$, $\hat \pi_t^n(\mathbbm{1}_A) \leq |A| + \frac{1}{n}$, where $|A|$ denotes the 
Lebesgue measure of $A$ and $\mathbbm{1}_\cdot$ is the usual characteristic function. 
Therefore, ${\hat u}_t(dx) = {\hat u}(t,x)dx$ for some random function
 $\{{\hat u}(t,x); t \in [0,T], x \in \mathbb R\}$, bounded between $0$ and $1$. 
Another observation is that $\{f(t);t \in [0,T]\}$ is Lipschitz, with Lipschitz constant bounded above 
by $|\beta-\alpha|$. For any fixed $n \in \mathbb N$, the measure $\hat \pi_t^n$ is the shift of $\pi_t^n$ by 
$\frac{x_t^n}{n}$. Since shifting by a continuous function is a continuous operation on the space 
$\mathcal M_+(\mathbb R)$, this relation is also satisfied by the limiting processes. 
Therefore, we have the relation ${\hat u}(t,x) = u(t,x+f(t))$ for any $t \in [0,T]$ and any $x \in \mathbb R$. 

Next, recall that the sum involved in equation \eqref{localRL} is equal to 
$\hat \pi_s^n(\epsilon^{-1}\mathbbm{1}_{(0,\epsilon]})$.
Since the step function $\epsilon^{-1}\mathbbm{1}_{(0,\epsilon]}$ is not a continuous function, 
we can not say that $\hat \pi_s^n(\epsilon^{-1}\mathbbm{1}_{(0,\epsilon]})$ converges.
However, since the limiting measure ${\hat u}(s,x)dx$ has a density and the step function
$\epsilon^{-1}\mathbbm{1}_{(0,\epsilon]}$ is {\em a.s.}~continuous, we have that
\[
\lim_{n \to \infty} \hat \pi_s^n(\epsilon^{-1}\mathbbm{1}_{(0,\epsilon]}) = 
\epsilon^{-1}\int_0^\epsilon {\hat u}(s,x) dx.
\]
On the other hand, by \eqref{Martingale},
\[
\int_0^t(1- 2\xi_s^n(0))ds = \frac{1}{\beta- \alpha} \Big(\frac{x_t^n}{n} - \widetilde{ {M}_t^n}\Big),
\]
which converges to $(\beta- \alpha)^{-1} f(t)$. 
Replacing these two limits into \eqref{localRL}, and using the identity ${\hat u}(t,x) = u(t,x+f(t))$ 
we obtain the relation
\[
\lim_{\epsilon \to 0} \mathbb E\Big[\Big| f(t) -\frac{\beta-\alpha}
{\epsilon}\int_0^t \int_0^\epsilon\big(1-2u(s,x+f(s))\big)dx ds \Big|\Big] =0.
\]
Since $f$ is uniformly Lipschitz and $u$ is smooth, the limit as $\epsilon \to 0$ of the integral 
above is equal to
\[
(\beta-\alpha)\int_0^t \big(1-2u(s,f(s))\big)ds.
\]
We conclude that $\{f(t); t \in [0,T]\}$ satisfies the integral equation
\[
f(t) = (\beta-\alpha) \int_0^t  \big(1-2u(s,f(s))\big)ds,
\]
which is nothing but the integral version of \eqref{edo}. 
Since this equation has a unique solution, we conclude that $f$ is deterministic and uniquely defined. 
This ends the proof of Theorem \ref{LLN}. 
Theorem \ref{Hydrolimit} follows from the argument, after recalling that $\hat u(t,x) = u(t,x+f(t))$ for all $x \in \mathbb R$, in particular $x=0$.

\section{Generalizations}
\label{s3}
In this section, we present some generalizations of Theorems \ref{LLN} and \ref{Hydrolimit}.
One possibility is to extend the results for other underlying dynamics. This is the content of Section \ref{s3.1}.
A second one is to consider different transitions for the random walk.
In Sections \ref{s3.2} and \ref{s3.3} we discuss how to generalize Theorems \ref{LLN} and \ref{Hydrolimit} 
for random walks with more general jump rates and with macroscopic jumps, respectively.

\subsection{The speed-change exclusion process}
\label{s3.1}
An example of underlying dynamics for which we can extend our results in a straightforward 
way is the so-called {\em speed-change} exclusion process, or stochastic lattice gas at infinite temperature.
In this dynamics, exchanges of particles between sites $z$ and $z+1$ are performed at 
rate $c_z(\eta)=c_0(\theta^{-z}\eta)$, where $c_0(\eta)$ is a strictly positive, local function 
which does not depend on the values of $\eta(0)$ and $\eta(1)$. More precisely, 
let $c_0: \Omega \to \mathbb R$ be a local, positive function. 
Notice that positivity plus locality imply that 
there exists a constant $\epsilon_0>0$ such that $\epsilon_0 \leq c_0(\eta) \leq \epsilon_0^{-1}$ for any 
$\eta \in \Omega$.
The generator of the speed-change simple exclusion process acts on local functions $f:
\Omega \to \mathbb R$ as
\begin{equation}
\label{g4}
\mathcal L^{\mathrm{sc}} f (\eta) \;=\; \sum_{z \in \mathbb Z}  c_{z}(\eta)\,
\{ f(\eta^{z,z+1}) - f(\eta) \} \;,
\end{equation}
where $c_z(\eta) = c_0(\theta^{-z} \eta)$. The stochastic evolution can be described as follows. At rate $c_{z}(\eta)$ the occupation variables 
$\eta(z)$, $\eta(z+1)$ are exchanged, and this rate depends on the occupation of the neighbors of $z,z+1$ 
up to some finite distance $R$. Such dependency gives raise to a non-linearity in the hydrodynamic limit. 
For this model, under the assumptions of Proposition \ref{p1}, the statement holds true 
with a Cauchy problem of the form
\[
\left\{
\begin{array}{rll}
\partial_t u(t,x) &= (1/2)\partial^2_{xx}\Phi( u(t,x)), &t \geq 0, x \in \mathbb R\\
u(0,x) &= u_0(x), &x \in \mathbb R\,,
\end{array}
\right.
\]
where $\Phi(\rho)$ is given in general by a variational formula (see Section 7 of \cite{KL99} for more details).

A choice which is very popular in the literature is $c_x(\eta) = 1 + a(\eta(x-1)+\eta(x+2))$ for some 
$a >-\tfrac{1}{2}$. For this particular choice, the model turns out to be gradient, see e.g. \cite{FL10}, and 
$\Phi(\rho) = \rho + a\rho^2$.

The generator of the environment as seen by the walker $x_t$ can be written as $\widetilde{\mathcal L}_n=n^2 \mathcal L^{\mathrm{sc}} +
n \mathcal L^{\mathrm{rw}}$, with $\mathcal L^{\mathrm{sc}}$ as in equation \eqref{g4}. We have the following version of Lemma \ref{DirBound}:

\begin{lemma}
For any density $f:\Omega \to \mathbb R$,
\[
\langle\sqrt f, \mathcal L_n \sqrt f\rangle \leq -\frac{n^2}{\epsilon_0} \mathcal D(f) + 2n|\alpha-\beta|.
\]
\end{lemma}

Keeping this lemma in mind, the proof of the replacement lemma, Theorem \ref{ReplaceLemma}, can be repeated 
for the speed-change exclusion process.
\emph{Mutatis mutandis}, the rest of the proofs is the same, as well the conclusions in the Theorem \ref{LLN} 
and the Theorem \ref{Hydrolimit}.

\subsection{Random walks with more general jump rates}
\label{s3.2}
For simplicity, the random walk we defined in Section \ref{s1.1} looks only at the state of the exclusion 
process at its current 
position. However, our result holds for a more general choice of the transition rates. 
Let $\gamma_z:\Omega\to[0,\infty),z\in\Z,$ be a collection 
of local functions prescribing the jump rates of the random walk, i.e.,
\begin{equation}\begin{aligned}\label{generalrates}
L f(\eta,x) &= \sum_{z \in \mathbb Z} \big[f(\eta^{z,z+1},x)-f(\eta,x)\big] + \sum_{z \in \mathbb Z} 
\gamma_z(\theta^{-x}\eta) \big[f(\eta,x+z)-f(\eta,x)\big].
\end{aligned}\end{equation}
Note that the setting discussed in the bulk of the paper corresponds to $\gamma_{+1}(\eta)=\beta+(\alpha-\beta)
\eta(0),
\gamma_{-1}(\eta)=\alpha+(\beta-\alpha)\eta(0)$, and $\gamma_z(\eta)=0$ else.

In the general case described in \eqref{generalrates}, by assuming that $\displaystyle\sum_{z\in\Z}\displaystyle\sup_{\eta}\tilde\gamma_z(\eta)<\infty$
, 
the hydrodynamic limit of Theorem \ref{Hydrolimit} still holds with
$\hat u(t,x)$ being the solution of  
\[
\left\{
\begin{array}{rll}
\partial_t \hat u(t,x) &= (1/2)\partial^2_{xx} \hat u(t,x) + \gamma(\hat u(t,0)) \partial_x \hat u(t,x), &t 
\geq 0, x \in \mathbb R\\
\hat u(0,x) &= u_0(x), &x \in \mathbb R ,
\end{array}\right.\]
where $\gamma(\rho)=\nu_\rho\left(\sum_z z\gamma_z\right)$.

In fact, with $V = (\xi(0)- \frac{1}{\epsilon n} \sum_{x = 1}^{\epsilon n} \xi(x))$ replaced by 
$V_{\epsilon n}=\sum_z z\gamma_z(\xi)-\gamma(\frac{1}{\epsilon n}\sum_{x=1}^{\epsilon n}\xi(x))$ the 
proofs generalize to this setting except for equation \eqref{eq:replacement}.
Here a more sophisticated argument is necessary. The proof follows from the arguments in \cite{JLS09}. 
Since the reference \cite{JLS09} is quite technical, we try to be more specific.
The proof of \eqref{eq:replacement} follows the celebrated one-block, two-blocks scheme introduced 
by \cite{GPV88}. The  one-block estimate reduces to prove that
\[
\lim_{\gamma \to \infty} \limsup_{\ell \to \infty} \limsup_{n \to \infty}  \sup_{f} \Big\{ \langle V_\ell,f\rangle - 
\frac{n}{\gamma}\mathcal D(f)\Big\} =0.
\]
This is equivalent to the estimation of Eq. (6.2) in \cite{JLS09}.
The two-blocks estimate reduces to 
prove that
\[
\lim_{\gamma \to \infty} \limsup_{\ell \to \infty} \limsup_{\varepsilon \to 0} \limsup_{n \to \infty}  
\sup_{f,x} \Big\{ \langle V_{\ell,x},f\rangle - \frac{n}{\gamma}\mathcal D(f)\Big\} =0, \text{ where}
\]
\[
V_{\ell,x}(\xi) = \gamma\Big(\frac{1}{\ell}\sum_{y=1}^{\ell}\xi(y)\Big)-\gamma\Big(\frac{1}{\ell}
\sum_{y=x+1}^{x+\ell}\xi(y)\Big)
\]
and the supremum is over densities $f: \Omega \to \mathbb R$ and over $2\ell+1 \leq x \leq \varepsilon n$. 
This is basically what is proven in Lemma 6.5 of \cite{JLS09}.
The rest of the proof follows like in the proof of Proposition 6.1 of \cite{JLS09}.

\begin{remark}
One of the main assumptions in \cite{JLS09} is a sharp lower bound on the spectral gap of the dynamics 
restricted to a finite box, which is well-known for the exclusion process.
\end{remark}

\subsection{Random walks with macroscopic jumps}
\label{s3.3}
For the hydrodynamic limit as presented to hold it is necessary that the rates of the random walk are 
properly rescaled. It is however possible to introduce rare large-scale jumps. 
Rescaling \eqref{generalrates}, we get
\begin{equation}\begin{aligned}\notag
L_n f(\eta,x) &= n^2\sum_{z \in \mathbb Z} \big[f(\eta^{z,z+1},x)-f(\eta,x)\big] + n\sum_{z \in \mathbb Z} 
\gamma_z(\theta^{-x}\eta) \big[f(\eta,x+z)-f(\eta,x)\big].
\end{aligned}\end{equation}
We can introduce long range jumps by adding a third term:
\begin{equation}\begin{aligned}\label{eq:longrange}
L_n f(\eta,x) &= n^2\sum_{z \in \mathbb Z} \big[f(\eta^{z,z+1},x)-f(\eta,x)\big] + n\sum_{z \in \mathbb Z} 
\gamma_z(\theta^{-x}\eta) \big[f(\eta,x+z)-f(\eta,x)\big] \\
&\quad+\sum_{z \in \mathbb Z} \tilde\gamma_z(\theta^{-x}\eta) \big[f(\eta,x+nz)-f(\eta,x)\big],
\end{aligned}\end{equation}
with 
\begin{equation}\label{JumpAss}
\sum_{z\in\Z}\sup_{\eta}\tilde\gamma_z(\eta)<\infty.\end{equation}
Note how the jump rates $\tilde\gamma_z:\Omega\to[0,\infty),z\in\Z,$ are not rescaled in time,
but the corresponding jumps are of order $n$. This leads to randomness in the hydrodynamic limit, 
where the random walk $x_t^n$ converges to a space-time inhomogeneous random walk $x_t$ on $\R$ with drift,
characterized by the generator
\begin{equation}\begin{aligned}\label{L_t}
 L^{\mathrm{rw}}_t f(x) &= \sum_{z\in\Z} \tilde\gamma_z(u(t,x)) [f(x+z)-f(x)] + \gamma(u(t,x))f'(x),\\
 \tilde\gamma_z(\rho) &= \nu_\rho(\tilde\gamma_z).
\end{aligned}\end{equation} 
The idea of the proof is rather straightforward, using the hydrodynamic limit without long-range jumps.
We now give the main lines of this proof.

Let $\tau^n$ be the time of the first macroscopic jump and $z^n$ the jump size over $n$, i.e., 
$$x^n_{\tau^n}=x^n_{\tau^{n}-}+nz^n\,.$$
The time $\tau^n$ is distributed according to the first arrival of any of the Poisson point processes 
$\{N_z^n: z\in\Z\}$ on $[0,\infty)$ with intensity measure $\{I^n_z: z\in\Z\}$ given by
\[ I^n_z([a,b])=\int_a^b \tilde\gamma_z(\xi^n_t)\,dt. \]
From the hydrodynamic limit without the macroscopic jumps, we know that $N_z^n$ converges to the 
limiting Poisson point process $N_z$ with intensity measures
\[ I_z([a,b]) = \int_a^b\tilde\gamma_z(\hat u(t,0))\,dt. \]

In Lemma \ref{tauconverge} below, we prove that for $t\in [0,\tau^n]$, $x^n_t$ converges to $x_t$, $t\in [0,\tau]$, 
where $x_t$ is the random walk described by the generator $L^{\mathrm{rw}}_t$ in \eqref{L_t},
and $\tau$ is the first macroscopic jump time, corresponding to the first arrival of the Poisson 
point processes $\{N_z:z\in\Z\}$. 

Finally, the result follows by iterating the same argument for the other macroscopic jumps after the first.

\begin{lemma}\label{tauconverge}
Define $X^n_{[0,\tau^n]}:=\{x^n_t :t\in [0,\tau^n]\}$ and $X_{[0,\tau]}:=\{x_t :t\in [0,\tau]\}$. 
Then, 
\[ X^n_{[0,\tau^n]} \text{ converges in distribution to } X_{[0,\tau]},\] 
as $n$ goes to infinity. 
\end{lemma}
\begin{proof}

The idea is to couple the Poisson point processes $\{N_z: z\in\Z\}$ and $\{N_z^n:z\in\Z\}$.
For $z\in\Z$, let $\bar{N}_z:=\left(\bar{N}_z(t)\right)_{t\geq 0}$ a Poisson process with rate 
$\|\tilde\gamma_z\|_{\infty}$. For $t\geq0$, consider the time-changed process
\begin{equation}\label{nClock}M_z^{n}(t):=\bar{N}_z\left( \int_0^t  \frac{\tilde\gamma_z(\xi^n_s)}
{\|\tilde\gamma_z\|_{\infty}} ds \right).\end{equation} 
Note that the jump times of this process have the same distribution as $N_z^n$.
Similarly, consider the time-changed process
\begin{equation}\label{uClock}
M_z(t):=\bar{N}_z\left( \int_0^t  \frac{\tilde\gamma_z(\hat u(s,0))}{\|\tilde\gamma_z\|_{\infty}} 
ds \right),\end{equation}
and note that its jump times have the same distribution as $N_z$. 
Hence, we can assume that there is a coupling under which the jump times of $N^n_z$ and $N_z$ are given by the 
jump times of \eqref{nClock} and \eqref{uClock}, respectively.

Note now that by definition, $\tau^n$ and $\tau$  are the first time that any of the processes
$\{M_z^n: z\in\Z\}$ and $\{M_z: z\in\Z\}$ have a jump, respectively.
Since the time-change in \eqref{nClock} converges to the one in \eqref{uClock}, $\tau^n$ 
converges to $\tau$ when conditioning on the realization of $\{\bar{N}_z:z\in\Z\}$.
Moreover, due to \eqref{JumpAss}, the jump events of $\{\bar{N}_z:z\in\Z\}$ are well-separated, 
consequently, the same holds true for the index $z^n$ of the first jump.
Therefore, by the hydrodynamic limit without macroscopic jumps, 
$x^n_{[0,\tau^n]}$ converges to $x_{[0,\tau]}$, 
and $x^n_{\tau^n}-x ^n_{\tau^{n}-}=nz^n$ converges to $z:=x_{\tau}-x_{\tau-}$.

It remains to show that the jump rates of $x_t$ match those given in \eqref{L_t}.
Observe that on the event that $\tau\geq t$, for arbitrary $\epsilon>0$, 
the probability of the occurrence of a jump of size 
$z$ before time $t+\epsilon$ is given by
\begin{equation}\begin{aligned}\notag
&P\left(M_z(t+\epsilon)-M_z(t)\geq1\right)=P\left(
M_z(t+\epsilon)-M_z(t)=1\right)+o(\epsilon)\\
&=1-\exp{\int_t^{t+\epsilon}\tilde\gamma_z(\hat u(s,0))ds} +o(\epsilon)=
\int_t^{t+\epsilon}\tilde\gamma_z(\hat u(s,0))ds +o(\epsilon).
\end{aligned}\end{equation}
Moreover, for $s<\tau$, $\hat u(s,0)=u(s,x_s)$.
Therefore $x^n_{[0,\tau^n]}$ indeed converges to $x_{[0,\tau]}$ with $x_t$ given by \eqref{L_t}.
\end{proof}


\end{document}